\documentclass[12pt, oneside, a4paper]{article}

\usepackage{amsmath}
\usepackage{amsfonts}
\usepackage{amssymb}
\usepackage{amsthm,mathrsfs}
\usepackage{enumerate}
\usepackage{graphicx}
\usepackage{subfigure}
\newtheorem{theorem}{Theorem}[section]
\newtheorem{lemma}[theorem]{Lemma}

\theoremstyle{definition}

\title{\textbf{$K_4$-free character graphs with diameter three}}
\author{Mahdi Ebrahimi\footnote{ m.ebrahimi.math@ipm.ir}
 \\
 {\small\em  School of Mathematics, Institute for Research in Fundamental Sciences (IPM)},\\{\small\em P.O. Box: 19395--5746, Tehran, Iran}\footnote{This research was supported by a grant from IPM.}}

\date{}

\begin{document}

\maketitle

%\fontsize{22}{23}\selectfont
%\baselineskip=12mm
\begin{abstract}
Let $G$ be a finite group and let $\rm{Irr}(G)$ be the set of all irreducible complex characters of $G$. Let $\rm{cd}(G)$ be the set of all character degrees of $G$ and denote by $\rho(G)$ the set of primes which divide some character degrees in $\rm{cd}(G)$. The character graph $\Delta(G)$ associated to $G$ is a graph whose vertex set is  $\rho(G)$ and there is an edge between two distinct primes $p$ and $q$ if and only if the product $pq$ divides some character degree of $G$. Suppose the character graph $\Delta(G)$ is $K_4$-free with diameter $3$.
 In this paper, we show that $|\rho(G)|\neq 5$, if and only if $G\cong J_1 \times A$, where $J_1$ is the first Janko's sporadic simple group and $A$ is abelian.

 \end{abstract}

\section{Introduction}
$\indent$ Let $G$ be a finite group. Also let ${\rm cd}(G)$ be the set of all character degrees of $G$, that is,
 ${\rm cd}(G)=\{\chi(1)|\;\chi \in {\rm Irr}(G)\} $, where ${\rm Irr}(G)$ is the set of all complex irreducible characters of $G$. The set of prime divisors of character degrees of $G$ is denoted by $\rho(G)$.

A useful way to study the character degree set of a finite group $G$ is to associate a graph to ${\rm cd}(G)$.
One of these graphs is the character graph $\Delta(G)$ of $G$ (see \cite{[I]}). Its vertex set is $\rho(G)$ and two vertices $p$ and $q$ are joined by an edge if the product $pq$ divides some character degree of $G$. We refer the readers to a survey by Lewis \cite{[M]} for results concerning this graph and related topics.

The diameter of $\Delta(G)$ denoted by $\rm{diam}(\Delta(G))$ is an important parameter to determine the structure of $G$.
 It was shown that $\rm{diam}(\Delta(G))\leqslant 3$ (see \cite{[DM]} and \cite{[P]}). When $\rm{diam}(\Delta(G))=3$,  $ \rho(G) $ can be partitioned  as $ \rho_{1} \cup \rho_{2} \cup \rho_{3} \cup \rho_{4} $ where no prime in $ \rho_{1} $ is adjacent to any prime in $ \rho_{3} \cup \rho_{4} $ and no prime in $ \rho_{4} $ is adjacent to any prime in $ \rho_{1}  \cup \rho_{2 } $, every prime in $  \rho_{2 } $ is adjacent to some primes in $ \rho_{3} $ and vice-versa, and $ \rho_{1}  \cup \rho_{2 } $ and $ \rho_{3} \cup \rho_{4} $ both determine complete subgraphs of $ \Delta(G)$ (see \cite{ams} and  \cite{[A]}). Casolo et al. \cite{[SDM]} has proved that if for a finite solvable group $G$,  $\Delta(G)$ is connected with diameter $3$, then there exists a prime $p$ such that $G=PH$, with $P$ a normal non-abelian Sylow $p$-subgroup of $G$ and $H$ a $p$-complement. The character graph of the first Janko's sporadic simple group $J_1$ is the only known example for $K_4$-free character graphs with diameter three. It motivates the following conjecture:\\

\noindent \textbf{Conjecture.} Suppose $G$ is a finite group. Then $\Delta(G)$ is a $K_4$-free graph with diameter three, if and only if $G\cong J_1 \times A$, where $A$ is abelian.\\

\noindent In this paper, we wish to solve this conjecture for the case $|\rho(G)|\neq 5$.\\

\noindent \textbf{Main Theorem.}  \textit{Let $G$ be a finite group and $\Delta(G)$ be a $K_4$-free graph with diameter $3$. Then $|\rho(G)|\neq 5$, if and only if $G\cong J_1 \times A$, where $A$ is abelian.}

%%%%%%%%%%%%%%%%%%%%%%%%%%%%%%%%%%%%%%
\section{Preliminaries}
$\indent$ In this paper, all groups are assumed to be finite and all
graphs are simple and finite. Note that when we require the structure of maximal subgroups and character table of simple groups, we will use Atlas of finite groups \cite{[At]}. The finite field with $q$ elements is denoted by $GF(q)$. For an integer $n\geqslant 1$, we use the notation $\pi(n)$, for the set of prime divisors of $n$. For a finite group $G$, the set of prime divisors of $|G|$ is denoted by $\pi(G)$. Also the solvable radical of $G$ is denoted by $R(G)$. We will use Zsigmondy's Theorem which can be found in \cite{zsi}.

Let $G$ be a finite group. If $H$ is a subgroup of $G$ and $\theta \in \rm{Irr}(H)$, we denote by $\rm{Irr}(G|\theta)$ the set of irreducible characters of $G$ lying over $\theta$ and define $\rm{cd}(G|\theta):=\{\chi(1)|\,\chi \in \rm{Irr}(G|\theta)\}$. The set of linear characters of $G$ denoted by $\hat{G}$ is an abelian group under the multiplication $\lambda_1\lambda_2(g):=\lambda_1(g)\lambda_2(g)$, where $g\in G$ and $\lambda_1,\lambda_2\in \hat{G}$. The order of a linear character $\lambda$ in $\hat{G}$ is denoted by $o(\lambda)$. Suppose $N\lhd G$ and $\theta$ is an irreducible character of $N$. The inertia subgroup of $\theta$ in $G$ is denoted by $I_G(\theta)$. We frequently use, Clifford's Theorem which can be found as Theorem 6.11 of \cite{[isa]} and Gallagher's Theorem which is Corollary 6.17 of \cite{[isa]}. We begin with Corollary 11.29 of \cite{[isa]}.

\begin{lemma}\label{fraction}
Let $ N \lhd G$ and $\varphi \in \rm{Irr}(N)$. Then for every $\chi \in \rm{Irr}(G|\varphi)$, $\chi(1)/\varphi(1)$ divides $[G:N]$.
\end{lemma}

\begin{lemma}\label{good}\cite{[Ton]}
Let $N$ be a normal subgroup of a group $G$ such that $G/N\cong S$, where $S$ is a non-abelian simple group. Let $\theta \in \rm{Irr}(N)$. Then either $\chi (1)/\theta(1)$ is divisible by two distinct primes in $\pi(G/N)$ for some $\chi \in \rm{Irr}(G|\theta)$ or $\theta$ is extendible to $\theta_0\in \rm{Irr}(G)$ and $G/N\cong A_5$ or $\rm{PSL}_2(8)$.
\end{lemma}

 \begin{lemma}\label{equality}\cite{hup2}
 Let $p$ and $r$ be two primes and let $f,m\geqslant 1$ be integers such that $p^f+1=r^m$. Then one of the following cases holds:\\
\textbf{ a)} $p=2$, $f=3$, $r=3$ and $m=2$;\\
 \textbf{b)} $p=2$, $m=1$, $r=2^f+1$ is a Fermat prime, and $f$ is a $2$-power;\\
\textbf{ c)} $r=2$, $f=1$, $p=r^m-1$ is a Mersenne prime, and $m$ is also a prime.
 \end{lemma}

Let $\Gamma$ be a graph with vertex set $V(\Gamma)$ and edge set
$E(\Gamma)$. If $E(\Gamma)=\emptyset$, $\Gamma$ is called an empty graph. The complement of $\Gamma$ and the induced subgraph of $\Gamma$ on $X\subseteq V(\Gamma)$
 are denoted by $\Gamma^c$ and  $\Gamma[X]$, respectively. We use the notation $\rm{diam}(\Gamma)$ for the diameter of $\Gamma$.  If $\Gamma_1,\Gamma_2,...,\Gamma_n$ are connected components of $\Gamma$, we use the notation $\Gamma=\Gamma_1\cup \Gamma_2\cup ...\cup \Gamma_n$ to determine the connected components of $\Gamma$.  Now let $\Delta$ be a graph with vertex set $V(\Delta)$. The join $\Gamma \ast \Delta$ of graphs $\Gamma$
and $\Delta$ is the graph $\Gamma \cup \Delta$ together with all
edges joining $V (\Gamma)$ and $V (\Delta)$. Also note that a complete graph with $n$ vertices is denoted by $K_n$. Now we state some relevant results on character graphs
needed in the next sections.

\begin{lemma}\label{part}\cite{ams}
Let $G$ be a finite non-solvable group with $\rm{diam}(\Delta(G))=3$. Then the complement of $\Delta(G)$ is by partite.
\end{lemma}

\begin{lemma}\label{path}\cite{zhang}
Suppose $G$ is a finite group. Then the character graph $ \Delta(G) $ is not a path of length four.
\end{lemma}

\begin{lemma} \label{trick}\cite{AC}
Let $G$ be a finite group. Also let $\mathcal{K}$ be any (non-empty) set of normal subgroups of $G$ isomorphic to  $\rm{PSL}_2(u^\alpha)$ or $\rm{SL}_2(u^\alpha)$, where $u^\alpha\geqslant 4$ is a prime power (possibly with different values of $u^\alpha$). Define $K$ as the product of all the subgroups in $\mathcal{K}$ and $C:=C_G(K)$. Then every prime $t$ in $\rho(C)$ is adjacent  in $\Delta(G)$ to all the primes $q$ (different from $t$) in $|G/C|$, with the possible exception of $(t,q)=(2,u)$ when $|\mathcal{K}|=1$, $K\cong \rm{SL}_2(u^\alpha)$ for some $u\neq 2$ and $Z(K)=P^\prime$, $P\in \rm{Syl}_2(C)$. In any case, $\rho(G)=\rho(G/C)\cup \rho(C)$.
\end{lemma}

There exists a classification of simple groups whose character graphs are $K_4$-free stated by Tong-viet.

\begin{lemma}\label{free}\cite{finite}
Let $S$ be a non-abelian simple group. Suppose that the character graph $\Delta(S)$ is $K_4$-free. Then one of the following cases holds:\\
\textbf{a)} $S\cong M_{11}\, or\, J_1$.\\
\textbf{b)} $S\cong A_n$ with $n\in\{5,6,8\}$.\\
\textbf{c)} $S\cong \rm{PSL}_2(q)$ with $q=p^f\geqslant 4$ and $|\pi (q\pm 1)|\leqslant 3$, where $p$ is prime.\\
\textbf{d) }$S\cong \rm{PSL}_3(q)$ with $q\in \{3,4,8\}$.\\
\textbf{e)}  $S\cong \rm{PSU}_3(q)$ with $q\in \{3,4,9\}$.\\
\textbf{f)}  $S\cong \rm{PSU}_4(2)$ or $^2B_2(q^2)$ with $q^2=2^3\, or\, 2^5$.
\end{lemma}

\begin{lemma}\label{her} \cite{[her]}
Let $G$ be a simple group. If $|\pi(G)|=3$, then $G$ is isomorphic to one of the groups $A_5,A_6,\rm{PSL}_2(7),\rm{PSL}_2(8),\rm{PSL}_2(17),\rm{PSL}_3(3),\rm{PSU}_3(3)$ and $\rm{PSU}_4(2)$.
\end{lemma}

The structure of the character graph of $PSL_2(q)$ is determined as follows:

\begin{lemma}\label{chpsl}\cite{[white]}
Let $G\cong \rm{PSL}_2(q)$, where $q\geqslant 4$ is a power of a prime $p$.\\
\textbf{a)}
 If $q$ is even, then $\Delta(G)$ has three connected components, $\{2\}$, $\pi(q-1)$ and $\pi(q+1)$, and each component is a complete graph.\\
\textbf{b)}
 If $q>5$ is odd, then $\Delta(G)$ has two connected components, $\{p\}$ and $\pi((q-1)(q+1))$.\\
\textbf{i)}
 The connected component $\pi((q-1)(q+1))$ is a complete graph if and only if $q-1$ or $q+1$ is a power of $2$.\\
\textbf{ii)}
  If neither of $q-1$ or $q+1$  is a power of $2$, then $\pi((q-1)(q+1))$ can be partitioned as $\{2\}\cup M \cup P$, where $M=\pi (q-1)-\{2\}$ and $P=\pi(q+1)-\{2\}$ are both non-empty sets. The subgraph of $\Delta(G)$ corresponding to each of the subsets $M$, $P$ is complete, all primes are adjacent to $2$, and no prime in $M$ is adjacent to any prime in $P$.
 \end{lemma}

%%%%%%%%%%%%%%%%%%%%%%%%%%%%%%%%%%%%%%%%%%%%%%
\section{About $\rm{PSL}_2(q)$}
$\indent$ In order to the proof of Main Theorem, we will need facts about $\rm{PSL}_2(q)$, where $q$ is a prime power.
 We will make use of Dickson's list of the subgroups of $\rm{PSL}_2(q)$, which can be found as Hauptsatz II.8.27 of \cite{hup}. We also use the fact that the Schur multiplier of  $\rm{PSL}_2(q)$ is trivial unless $q=4$ or $q$ is odd, in which case it is of order $2$ if $q\neq 9$ and of order $6$ if $q=9$.
 In the sequel of this section, we let $t$ be a prime $f\geqslant 1$ be an integer, $q=t^f$ and $S\cong \rm{PSL}_2(q)$.

\begin{lemma}\label{cdp}\cite{[non]}
Let $S\cong \rm{PSL}_2(q)$, where for some prime $p$, $q=2^p$ and $G:=\rm{Aut}(S)$. Then $\rm{cd}(G)=\{1,q-1,q,(q-1)p,(q+1)p\}$.
\end{lemma}

 \begin{lemma}\label{lw}\cite{[non]}
 Let $q\geqslant 5$, $f\geqslant 2$ and $q\neq9$. If $S\leqslant G\leqslant \rm{Aut}(S)$, then $G$ has irreducible characters of degrees $(q+1)[G:G\cap \rm{PGL}_2(q)]$ and $(q-1)[G:G\cap \rm{PGL}_2(q)]$.
 \end{lemma}

\begin{lemma} \label{interest}
 Assume that $q=2^f$, for some $\epsilon \in \{\pm 1\}$, $|\pi(q+\epsilon)|=1$ and $|\pi(q-\epsilon)|=2\; or\; 3$. Then one of the following cases holds:\\
 \textbf{a)} $f$ is a prime, $q-1$ is a Mersenne prime and $|\pi(q+1)|=2\; or\; 3$.\\
 \textbf{b)} $f=4$, $q-1=3.5$ and $q+1=17$.\\
 \textbf{c)} $f=8$, $q-1=3.5.17$ and $q+1=257$.
 \end{lemma}
 \begin{proof}
 If $\epsilon=-1$, then using lemma \ref{equality}, $f$ is a prime, $q-1$ is a Mersenne prime  and $|\pi(q+1)|=2\; or \; 3$. Now we assume that $\epsilon=1$. By Lemma \ref{equality}, $f$ is a 2-power and $2^f+1=q+1$ is a Fermat prime.
 Therefore as $|\pi (q-1)|=2\; or\;3$, $f=4\;or\;8$ and the proof is completed.
 \end{proof}

 \begin{lemma}\label{evenfive}\cite{commu}
 If $q$ is even and $|\pi(q\pm 1)|=2$, then either $f$ is a prime or $f=6$ or $9$.
 \end{lemma}

Suppose $G$ is a finite group , $q\geqslant 11$ and $G/R(G)=S$, $\theta\in \rm{Irr}(R(G))$, $I:=I_G(\theta)$ and $N:=I/R(G)$. Now we peresent some results on $\rm{cd}(G|\theta)$.

\begin{lemma}\cite{commu}\label{Frobenius}
Suppose $N$ is a Frobenius group whose kernel is an elementary abelian $p$-group. Then one of the following cases holds:\\
\textbf{a)} $\theta$ is extendible to $I$ and $\rm{cd}(G|\theta)=\{\theta(1)[G:I],\theta(1)b\}$, for some positive integer $b$ divisible by $(q^2-1)/(2,q-1)$.\\
\textbf{b)} $\theta$ is not extendible to $I$ and all character degrees in $\rm{cd}(G|\theta)$ are divisible by $p(q+1)\theta(1)$.
\end{lemma}

\begin{lemma}\label{general}\cite{commu}
Let $N\cong \rm{PGL}_2(p^m)$, where $2m$ is a positive divisor of $f$ and  $p^m\neq 3$. Then for some $m_0\in \rm{cd}(G|\theta)$, $m_0$ is divisible by $\theta(1)p^{f-m}(p^f+1)$.
\end{lemma}

\subsection{On irreducible $\rm{PSL}_2(q)$-modules over $GF(p)$}
$\indent$ Let $F\subseteq E$ be a field extension and let $D$ be an $F$-representation of a finite group $G$. Then $D$ maps $G$ into a group of nonsingular matrices over $F$ that, of course, are also non-singular over $E$. Therefore $D$ is viewed as an $E$-representation of $G$. As such it is denoted by $D^E$. If $D_1$ and $D_2$ are similar $F$-representations, then $D^E_1$ and $D^E_2$ are similar and it follows that if $D$ corresponds to the $F[G]$-module $V$, then there exists a uniquely defined (up to isomorphism) $E[G]$-module $V^E$ that corresponds to $D^E$. Note that $V^E\cong V \otimes _FE$. Now suppose $\chi$ is an $E$-character of  $G$. The subfield  of $E$ generated by $F$ and the character values $\chi(g)$ for $g\in G$, is denoted by $F(\chi)$. We refer to \cite{[isa]} for a thorough analysis of this and related topics.

Now we wish to determine the classification of irreducible $\rm{SL}_2(2^f)$-modules over $GF(2^f)$, where $f\geqslant 3$ is an integer. This can be done with Steinberg's tensor product Theorem (see Corollary 3.17 of \cite{rep}). First, consider the natural $2$-dimensional module $V_0$ of $\rm{SL}_2(2^f)$ over $GF(2^f)$. This has $f$ "Galois twists" $V_i$, $i=0, \dots, f-1$, which are obtained as follows: Let $\sigma$ denote a Galois automorphism of $GF(2^f)$ of order $f$ defined by $\sigma(a):=a^2$, $a\in GF(2^f)$. Set $\sigma_i:=\sigma^i$, $i=0, \dots, f-1$. Then $V_i$ is the module with the same underlying vector space as $V_0$, but an element $x=[a_{st}]_{2\times2}$ of $\rm{SL}_2(2^f)$ acts on $V_i$ as $\sigma_i(x):=[\sigma_i(a_{st})]_{2\times 2}$. For every non-empty subset $J$ of $\{0, \dots, f-1\}$ let $V_J$ denote the tensor product of the $V_j$'s over $GF(2^f)$, for $j$ in $J$. For $J$ the empty set, let $V_J$ denote the trivial module. Notice that for $J=\{j\}$, we have $V_J=V_j$. We thus obtain $2^f$ modules $V_J$, $J\subseteq \{0,\dots,f-1\}$. By Steinberg's tensor product Theorem, these are pairwise non-isomorphic irreducible $GF(2^f)[\rm{SL}_2(2^f)]$-modules, and every irreducible module of $\rm{SL}_2(2^f)$ over $GF(2^f)$ is isomorphic to some such $V_J$. Note that the Galois automorphism of $GF(2^f)$ acts
on the set $S:=\{V_J|J\subseteq \{0, \dots, f-1\}\}$ by twisting. With respect to this action, we denote the orbit of $J\subseteq \{0, \dots, f-1\}$ by $\bar{J}$.

\begin{lemma}\label{tensor}
Let $V$ be a non-trivial irreducible $\rm{SL}_2(2^f)$-module over $GF(2)$, where $f\geqslant 3$ is an integer. Then there exists $J\subseteq \{0, \dots, f-1\}$ such that $V^{GF(2^f)}\cong \bigoplus_{I\in \bar{J}} V_I$, where $V_I$ is defined as above.
\end{lemma}

\begin{proof}
Let $F:=GF(2)$ and $E:=\bar{F}$ be the algebraic closure of $F$. Also let $\chi$ be an irreducible character afforded by an irreducible constituent $W$ of $V^E$. Then it is easy to see that $L:=F(\chi)$ is equal to $GF(2^f)$. Now suppose $Z$ is an irreducible $L[\rm{SL}_2(2^f)]$-module such that $W$ is a constituent of $Z^E$. By Corollary $9.23$ of \cite{[isa]}, $W$ is similar to $Z^E$. Hence using Theorem $9.21$ of \cite{[isa]} for the field extension $F\subseteq L \subseteq E$, we are done.
\end{proof}

 Let  $F\subseteq E$ be a field extension and $V$ be a $G$-module  over the field $F$. The centralizer of $v\in V$ in $G$ is defined as $C_G(v):=\{g\in G|\, gv=v\}$. Note that for every $v\otimes a\in V^E$, $C_G(v\otimes a)=C_G(v)$.

\begin{lemma}\label{center}
Suppose $V$ is a non-trivial irreducible $\rm{SL}_2(2^f)$-module over $GF(2)$, where $f\geqslant 3$ is a positive integer.  Then there exists $v\in V$ so that $C_G(v)= 1$.
\end{lemma}

\begin{proof}
Using Lemma \ref{tensor}, there exists $J\subseteq \{0, \dots, f-1\}$ such that $V^{GF(2^f)}\cong \bigoplus_{I\in \bar{J}} V_I$. Since $V$ is non-trivial, $J$ is non-empty. For every $i\in \{0, \dots, f-1\}$, let $\{e_1^i, e_2^i\}$ be the standard basis for $V_i$. Then , the set $\{\bigotimes_{j\in J} e_{\epsilon_j}^j|\, \epsilon_j \in \{1,2\}\}$ is a basis for $V_J$. We fix $j_0\in J$. If $J=\{j_0\}$, then we choose $v:=e_1^{j_0}+e_{2}^{j_0}$ and it is easy to see that $C_{\rm{SL}_2(2^f)}(v)= 1$. Now let $J-\{j_0\}\neq \emptyset$. Without loss of generality, we can assume that for every $j\in J$, $j_0\leqslant j$. Set \\ $v:=\bigotimes_{j\in J} e_{1}^j+\bigotimes_{j\in J} e_{2}^j + e_1^{j_0}\otimes(\bigotimes_{j\in J-\{j_0\}} e_{2}^j)+ e_2^{j_0}\otimes(\bigotimes_{j\in J-\{j_0\}} e_{1}^j)$. \\ Then we can see that $C_{\rm{SL}_2(2^f)}(v)= 1$ and it completes the proof.
\end{proof}

Let $H$ and $V$ be finite groups, and assume that $H$ acts by automorphisms on $V$. Given a prime number $q$, the pair $(H,V)$ is said a pair  satisfying $\mathcal{N}_q$ if $q$  divides $[H:C_H(V)]$ and, for every non-trivial $v\in V$, there exists a Sylow $q$-subgroup $Q$ of $H$ such that $Q\unlhd C_H(v)$.

\begin{lemma}\label{irreducible}\cite{AC}
Assume that the pair $(H,V)$ satisfies $ \mathcal{N}_q$. Then $V$ is an elementary abelian $r$-group for a suitable prime $r$, and it is an irreducible $H$-module over $GF(r)$.
\end{lemma}

We end this subsection with a useful result.

\begin{lemma}\label{module}\cite{AC}
Let $H\cong \rm{PSL}_2(r)$, where $r\geqslant 5$ is a prime power. Then for any odd prime $q$, there does not exist any $H$-module $M$ such that $(H,M)$ satisfies $\mathcal{N}_q$.
\end{lemma}

\subsection{Almost simple groups with socle $\rm{PSL}_2(q)$}
$\indent$ In this subsection we wish to state some results on almost simple groups with socle $\rm{PSL}_2(q)$ where $q$ is a prime power.
\begin{lemma}\label{otf}
Suppose $G$ is a finite group such that $S:=G/R(G)\cong \rm{PSL}_2(q)$, where for some integer $f\geqslant 4$, $q:=2^f$. If $\theta\in \rm{Irr}(R(G))$, $I:=I_G(\theta)$ and $N:=I/R(G)$, then one of the following cases occurs:\\
\textbf{a)} $N=S$ and $\rm{cd}(G|\theta)=\{m\theta(1)|\,m\in \rm{cd}(S)\}$.\\
\textbf{b)} $N$ is an elementary abelian $2$-group and every $m\in\rm{cd}(G|\theta)$ is divisible by $(q^2-1)\theta(1)$. \\
\textbf{c)} $N$ is contained in a dihedral group and every $m\in\rm{cd}(G|\theta)$ is divisible by $2(q-1)\theta(1)$ or $2(q+1)\theta(1)$.\\
\textbf{d)} $N$ is a Frobenius group whose kernel is an elementary abelian $2$-group and either all character degrees in $\rm{cd}(G|\theta)$ are divisible by $2(q+1)\theta(1)$ or some $m\in\rm{cd}(G|\theta)$ is divisible by $(q^2-1)\theta(1)$.\\
\textbf{e)} $N\cong A_5$ and $\Delta(G)[(\pi(S)\cup \pi(\theta(1)))-\{5\}]$ is a complete graph.\\
\textbf{f)} $N\cong \rm{PSL}_2(2^m)$, where $m\neq f$ is a positive divisor of $f$ and $2^m\neq 4$. In this case, $2^{f-m}\theta(1)(2^m\pm 1)(2^{2f}-1)/(2^{2m}-1) \in \rm{cd}(G|\theta)$.
\end{lemma}

\begin{proof}
 We use Dickson's list to determine the structure of $N$. If $N=S$, then as the Schur multiplier of $S$ is trivial, Gallagher's Theorem implies that $\rm{cd}(G|\theta)=\{m\theta(1)|\,m\in \rm{cd}(S)\}$. Also if $N$ is either an elementary abelian $2$-group or contained in a dihedral group, then using Clifford's Theorem, we are done. Next, if  $N$ is a Frobenius group whose kernel is an elementary abelian $2$-group, by Lemma \ref{Frobenius}, we have nothing to prove. Now let $N\cong A_5$. Since $\rm{SL}_2(5)$ is the Schur representation of $A_5$, $m:=\theta(1) q(q^2-1)/20$ or $n:=\theta(1) q(q^2-1)/10$ is a character degree in $\rm{cd}(G|\theta)$. Hence $\Delta(G)[(\pi(S)\cup \pi(\theta(1)))-\{5\}]$ is a complete graph and we are done. Finally, if  $N\cong \rm{PSL}_2(2^m)$, where $m\neq f$ is a positive divisor of $f$ and $2^m\neq 4$, then as the Schur  multiplier of $\rm{PSL}_2(2^m)$ is trivial, using Clifford's Theorem, $b:=2^{f-m}\theta(1)(2^m\pm 1)(2^{2f}-1)/(2^{2m}-1) \in \rm{cd}(G|\theta)$. It completes the proof.
 \end{proof}

Now we state a useful assertion which will be required in the next sections.

\begin{lemma}\label{power}
Let $G$ be a finite perfect group, $f\geqslant 4$ be an integer, $q:=2^f$ and $S:=G/R(G)\cong \rm{PSL}_2(q)$. If  either\\
\textbf{a)}
There exists $\epsilon \in \{\pm 1\}$ such that in $\Delta(G)$, no prime in $\pi(2(q+\epsilon))$ is adjacent to any prime in $\pi(q-\epsilon)$, or\\
 \textbf{b)} $|\pi(S)|=5\, or \, 6$ and $\Delta(G)$ is $K_4$-free.\\
 Then $G\cong \rm{PSL}_2(q)$.
\end{lemma}

\begin{proof}
On the contrary, assume that $G$ is not simple and hence $R(G)$ is non-trivial. If $\Delta(G)$ has three connected components, then by Theorem 4.1 of \cite{[nsc]}, $G$ is simple and we obtain a contradiction. Next, suppose $\Delta(G)$ has two connected components. Then as $G$ is not simple, by Lemma 5 of \cite{Tong}, $G/L \cong \rm{SL}_2(q_1)$ and $G$ has an irreducible character of degree $q^2_1-1$, where $L$ is an elementary abelian subgroup of $G$ and $q_1\geqslant 4$ is a prime power. Hence as $q$ is even, $q=q_1$ and $\Delta(G)[\pi(q^2-1)]$ is a complete graph which is a contradiction. Therefore $\Delta(G)$ is connected. Assume that $R(G)$ is non-abelian. Then $R(G)^\prime$ is non-trivial and thus $G/R(G)^\prime$ satisfies the hypothesis of the lemma with smaller order, hence $G/R(G)^\prime \cong S$. Therefore $R(G)=R(G)^\prime$ and since $R(G)$ is solvable, we must have $R(G)=1$. It is a contradiction. Thus $R(G)$ is abelian. Since $\Delta(G)$ is connected, we have $|\pi(S)|=5\, or \, 6$ and $\Delta(G)$ is $K_4$-free. Let $1\neq \lambda \in \rm{Irr}(R(G))$, $I:=I_G(\lambda)$ and $N:=I/R(G)$.
 As the Schur multiplier of $S$ is trivial and $G$ is perfect, $N< S$. Since $\Delta(G)$ is $K_4$-free, $\Delta(S)$ is too. Hence by Lemma \ref{chpsl}, $\Delta(S)\cong K_1\cup K_1\cup K_3$, $K_1\cup K_2\cup K_2$ or  $K_1\cup K_2\cup K_3$.  Thus using this fact that $\Delta(G)$ is $K_4$-free, by Lemmas \ref{interest}, \ref{evenfive} and \ref{otf}, we deduce that $N$ is either a non-trivial cyclic group of odd order, a dihedral group or a Frobenius group whose kernel is an elementary abelian $2$-group.
 Now we claim that $R(G)$ is a $2$-group. On the contrary, we assume that $R(G)$ is not a $2$-group. Then $H:=\{\lambda \in \rm{Irr}(R(G))|\, o(\lambda)\, is\,odd\}$ is non-trivial. Let $1\neq \lambda \in H$, $I:=I_G(\lambda)$ and $N:=I/R(G)$. If $N$ is a Frobenius group whose kernel is an elementary abelian $2$-group, then we deduce that $\lambda$ is extendible to $I$ and so by Lemma \ref{Frobenius}, $\Delta(G)[\pi(q^2-1)]$ is a complete graph. It is a contradiction as $\Delta(G)$ is $K_4$-free. Hence $N$ is either a non-trivial cyclic group of odd order or a dihedral group. Thus as $\Delta(G)$ is $K_4$-free, there exists $\epsilon \in \{\pm 1\}$ such that for some odd prime $s\in \pi (q+\epsilon)$, for every $1\neq \lambda \in H$, we have $I_G(\lambda)/R(G)$ contains a $s$-Sylow subgroup $Q$ of $S$ so that $Q\lhd I_G(\lambda)/R(G)$. Hence
  the pair $(S,H)$ satisfies $\mathcal{N}_s$ and using Lemmas \ref{irreducible} and \ref{module}, we obtain a contradiction. Thus $R(G)$ is an abelian $2$-group. Therefore $W:=\{\lambda \in \rm{Irr}(R(G))|\, \lambda^2=1\}$ can be viewed as a $S$-module over $GF(2)$. Let $V$ be an irreducible $S$-submodule of $W$ over $GF(2)$. Since $G$ is perfect, $V$ is non-trivial. Thus by Lemma \ref{center}, there exists a non-trivial element $\lambda_0\in V$ such that $I_G(\lambda_0)/R(G)= C_S(\lambda_0)= 1$. Hence for every $m\in \rm{cd}(G|\lambda_0)$, we have $|\pi(m)|\geqslant 4$ and it is a contradiction. Thus $G\cong \rm{PSL}_2(q)$ and the proof is completed.
\end{proof}

\begin{lemma}\label{direct product}
Let $G$ be a finite group, $f\geqslant 4$ be an integer and $q:=2^f$. Also let for some $R(G)<M\triangleleft G$, $G/R(G)$ be an almost simple group with socle $S:=M/R(G)\cong \rm{PSL}_2(q)$. If either\\
\textbf{a)}
There exists $\epsilon \in \{\pm 1\}$ such that in $\Delta(G)$, no prime in $\pi(2(q+\epsilon))$ is adjacent to any prime in $\pi(q-\epsilon)$, or\\
 \textbf{b)} $|\pi(S)|=5\, or \, 6$ and $\Delta(G)$ is $K_4$-free.\\
 Then $\rm{diam}(\Delta(G))\leqslant 2$.
\end{lemma}

\begin{proof}
 Let $H$ be the last term of the derived series of $M$. Since $M$ is non-solvable, $H$ is non-trivial. Let $N:=H\cap R(G)$. As $H/N \cong HR(G)/R(G)\lhd M/R(G)\cong  S$, we deduce that $H/N\cong S$. Hence using Lemma \ref{power}, $N=1$ and $H\cong \rm{PSL}_2(q)$. Let $C:=C_G(H)$. It is easy to see that $R(G)\subseteq C$. Therefore using Lemmas \ref{trick} and \ref{lw}, we deduce that $\rm{diam}(\Delta(G))\leqslant 2$.
  \end{proof}

    We end this section with the following result.
 \begin{lemma}\label{sen}
Suppose $G$ is a finite group such that $S:=G/R(G)\cong \rm{PSL}_2(q)$ where for some odd prime $p$ and integer $f\geqslant 1$, $q:=p^f$, and $|\pi(S)|\geqslant 4$. Also assume that $\theta\in \rm{Irr}(R(G))$, $I:=I_G(\theta)$, $N:=I/R(G)$ and when $|\pi(S)|\leqslant 5$, $\theta(1)$ is divisible by some $r\in \rho(G)-\pi(S)$. If $\Delta(G)$ is $K_4$-free and $N\neq S$, then  $\Delta(S)\cong K_1\cup K_3$ and $N$ is either a Frobenius group whose kernel is an elementary abelian $p$-group and $\theta$ is not extendible to $I$, or contained in a dihedral group.
\end{lemma}

\begin{proof}
On the contrary, we assume that the statement is not true. Then using Dickson's list, one of the following cases occurs:\\
Case 1.  $N$ is either a Frobenius group whose kernel is an elementary abelian $p$-group, and $\theta$ extends to $I$, or an elementary abelian $p$-group. Then using Lemma \ref{Frobenius} and Clifford's Theorem, there exists $m\in \rm{cd}(G|\theta)$ such that $m$ is divisible by $\theta(1)(q^2-1)/2$. Hence $\Delta(G)[\pi(m)]$ contains a copy of $K_4$ which is a contradiction.\\
Case 2.
$\Delta(S)\ncong K_1\cup K_3$ and $N$ is either a Frobenius group whose kernel is an elementary abelian $p$-group and $\theta$ does not extend to $I$, or  contained in a dihedral group. Then using Lemma \ref{Frobenius} and Clifford's Theorem, we can see that for some $m\in \rm{cd}(G|\theta)$, $m$ is divisible by either $\theta(1)p(q-1)$ or $\theta(1)p(q+1)$. Hence $\Delta(G)[\pi(m)] $ has a copy of $K_4$ and it is a contradiction.\\
Case 3. $N\cong A_5$. Since $\rm{SL}_2(5)$ is the Schur representation of $A_5$, there exist $m,n\in \rm{cd}(G|\theta)$ such that $m=q(q^2-1)\theta(1)/30$ and $n$ is divisible by  $q(q^2-1)\theta(1)/40$. Now we consider the following subcases:\\
a) $p=3$. Clearly $f\geqslant 4$ and hence $|\pi(m)|\geqslant 4$. It is a contradiction.\\
b) $p\neq 3$. Then as $p$ is odd and $N\cong A_5$, $2$ and $3$ are adjacent vertices in $\Delta(S)\subseteq \Delta(G)$. Therefore the induced subgraph of $\Delta(G)$ on $\pi(m)\cup \pi(n)$ has a copy of $K_4$ which is a contradiction.\\
Case 4.  $N$ is isomorphic to either $S_4$ and $16$ divides $q^2-1$, or $A_4$. By Clifford's Theorem, we can see that there exists $m \in \rm{cd}(G|\theta)$ so that $m$ is divisible by one of the values $\theta(1)q(q^2- 1)/|N|$ or $3\theta(1)q(q^2- 1)/2|N|$. Clearly, $|\pi(m)|\geqslant 4$ and we have a contradiction.\\
Case 6. $N \cong \rm{PSL}_2(p^m)$, where $m\neq f$ is a positive divisor of $f$, $p^m\geqslant 7$ and $p^m\neq 9$. There exists an integer $n>1$ so that $f=mn$.
 Now one of the following subcases occurs:\\
a) $n$ is odd. Then as $\rm{SL}_2(p^m)$ is the Schur representation of $\rm{PSL}_2(p^m)$, by Clifford's Theorem, $b:=\theta(1)p^{f-m}(p^f-1)(p^f+1)/(p^m+1)\in \rm{cd}(G|\theta)$. Using Zsigmondy's Theorem, there exists a prime $t$ so that $t| p^f+1$ and $t\nmid p^m+1$. Hence  $ |\pi(tp(p^f-1)\theta(1))|\geqslant 4$ which is a contradiction.\\
b) $n$ is even. Then $p^{2m}-1|p^f-1$. By Clifford's Theorem, $d:=\theta(1)p^{f-m}(p^m+1)(p^f+1)(p^{f}-1)/(p^{2m}-1) \in \rm{cd}(G|\theta)$. Using Zsigmondy's Theorem, there exists a prime $t$ so that  $t|p^f+1$ and $t\nmid p^m+1$. Therefore $|\pi(p(p^m+1)(p^f+1)\theta(1))|\geqslant 4$ and we again obtain a contradiction.\\
 Case 7. $N\cong \rm{PSL}_2(9)$, $p=3$, $f$ is even and $f\neq 2$. Then there exists an integer $n>1$ so that $f=2n$. Looking at the character table of the Schur representation of $A_6$, we can see that for some $m\in \rm{cd}(G|\theta)$, $m$ is divisible by  $2.3^{2(n-1)}(3^{4n}-1)\theta(1)/(3^{4}-1)$. Hence we can see that $|\pi(m)|\geqslant 4$ and it is a contradiction.\\
 Case 8. $N\cong \rm{PGL}_2(p^m)$, where $2m$ is a positive divisor of $f$ and $p^m\neq 3$.  By Lemma \ref{general}, some $m_0\in \rm{cd}(G|\theta)$ is divisible by $\theta(1)p^{f-m}(p^{f}+1)$. Note that by Lemma \ref{equality}, $|\pi(p^f+1)|\geqslant 2$. Hence we can see that $|\pi(m_0)|\geqslant 4$ which is a contradiction.
\end{proof}

%%%%%%%%%%%%%%%%%%%%%%%%%%%%%%%%%%%%%%

\section{Proof of Main Theorem}
$\indent$ In this section, we wish to prove our main result. Note that when $G\cong J_1\times A$, where $A$ is abelian, then it is clear that $|\rho(G)|\neq 5$. We begin with a useful observation.

\begin{lemma}\label{gej}
Let $G$ be a finite group and $G/R(G) \cong J_1$. If $\Delta(G) \cong \Delta(J_1)$, then $R(G)$ is abelian and $G\cong J_1\times R(G)$.
\end{lemma}
\begin{proof}
Let $H$ be the last term of the derived series of $G$. Since $G$ is non-solvable, $H$ is perfect. Suppose $N:=H\cap R(G)$.  Then as $G/R(G)$ is a non-abelian simple group, $H/N\cong J_1$. Hence $\Delta(H) \cong \Delta(J_1)$. We claim that $N=1$. Suppose on the contrary that $N$ is non-trivial. Since $N$ is solvable, it has a non-trivial linear character $\lambda \in \rm{Irr}(N)$.
 Also as the Schur multiplier of $J_1$ is trivial,  we deduce that $\lambda$ is not $H$-invariant and so the inertia subgroup $I:=I_H(\lambda)$ is a proper subgroup of $H$. Let $M$ be a maximal subgroup of $H$ containing $I$. Then as $\Delta(H)= \Delta(J_1)$, Clifford's Theorem implies that $M/N\cong \rm{PSL}_2(11)$ or $\mathbb{Z}_2\times A_5$. If $I$ is contained in a maximal subgroup $M^\prime$ of $M$, then using Clifford's Theorem, for some $m\in \rm{cd}(H|\lambda)$, $|\pi(m)|\geqslant 4$ and we have a contradiction. Thus $I=M$. If $M/N\cong \rm{PSL}_2(11)$, then as $\rm{SL}_2(11)$ is the Schur representation of $M/N$, $M:=2^2.5.7.19\in \rm{cd}(H|\lambda)$ which is impossible. Hence $M/N\cong \mathbb{Z}_2\times A_5$. If $\lambda$ extends to $I$, then by Gallagher's Theorem, $M:=5.7.11.19\in \rm{cd}(H|\lambda)$. It is a contradiction. Therefore $\lambda$ does not extend to $I$. Let $\varphi\in \rm{Irr}(M|\lambda)$. Then $\varphi(1)\neq 1$ and by Lemma \ref{fraction},  $\varphi(1)$ divides $|M/N|$. Hence as $[H:M]$ and $|M/N|$ are co-prime, Clifford's Theorem implies that $\varphi^H (1)\in \rm{cd}(H|\lambda)$ is divisible by at least four distinct primes and it is a contradiction.
Thus $N=1$ and $G\cong H\times R(G)$. If $R(G)$ is non-abelian and $p \in \rho(R(G))$, then the degree of vertex $p$ in $\Delta(G)$ is five which is impossible. Thus $G\cong J_1\times R(G)$, where $R(G)$ is abelian.
\end{proof}

\begin{lemma}\label{diamtwo}
Let $G$ be a finite group, $ C\lhd G$, $S$ be a non-abelian simple group and for some integer $k\geqslant 1$, $S^k\unlhd G/C\leqslant \rm{Aut}(S^k)$. If $\Delta(G)[\pi(G/C)]\cong K_3$, then $\rm{diam}(\Delta (G))\leqslant 2$.
\end{lemma}

\begin{proof}
 Let $L\geqslant C$ be a subnormal subgroup of $G$ such that $L/C \cong S$. Clearly, $\Delta(L)$ is a subgraph of $\Delta(G)$. Now let $\pi:=\rho(G)-\rho(G/C)$. Then by Lemma \ref{fraction}, we deduce that $\pi \subseteq \rho(C)$. Hence using Lemma \ref{good}, for every prime $p\in\pi$, there exist two distinct primes  $q_1,q_2\in\pi(L/C)=\pi(S)$ such that $\Delta(G)[\pi(pq_1q_2)]\cong K_3$. Thus as $\rho(G)=\pi(S) \cup \pi$ and $|\pi (S)| = 3$, we can see that $\rm{diam}\Delta(G)\leqslant 2$.
\end{proof}

\begin{lemma}\label{as}
Let $G$ be a finite group such that $|\rho(G)|\neq 5$ and $\Delta(G)$ is a $K_4$-free graph with diameter three. Then $|\rho(G)|=6$ and there exists a normal subgroup $R(G)<M\leqslant G$ so that $G/R(G)$ is an almost simple group with socle $M/R(G)$.
\end{lemma}

\begin{proof}
Since $\Delta(G)$ is $K_4$-free, Lemmas \ref{part} and \ref{path} imply that $|\rho(G)|= 6$. By Theorem 4 of \cite{[A]}, $G$ is non-solvable.  Let $M/R(G)$ be a chief factor of $G$. Since $R(G)$ is the largest normal solvable subgroup of $G$, $M/R(G)$ must be non-abelian and thus $M/R(G)\cong S^k$, for some non-abelian simple group $S$ and some integer $k\geqslant 1$. We show that $k=1$. On the contrary, assume that $k>1$. Let $C/R(G)=C_{G/R(G)}(M/R(G))$. Then $S^k\cong MC/C\lhd G/C\leqslant \rm{Aut}(MC/C)$ and so $G/C$ satisfies the hypothesis of Main Theorem of \cite{[LMc]}. Therefore $\Delta(G/C)$ is a complete subgraph of $\Delta(G)$ with at least $3$ vertices. Since $\Delta(G)$ is $K_4$-free, we obtain $|\rho (G/C)|=3$. Therefore $\Delta(G/C)$ is a triangle. Hence using Lemma \ref{diamtwo}, $\rm{diam}(\Delta(G))\leqslant 2$ which is a contradiction. Thus $M/R(G)\cong S$ is a non-abelian simple group.

 Let $C/R(G)=C_{G/R(G)}(M/R(G))$. We claim that $C=R(G)$ and $G/R(G)$ is an almost simple group with socle $M/R(G)$. Suppose on the contrary that $C\neq R(G)$ and let $L/R(G)$ be a chief factor of $G$ with $L\leqslant C$. Similar to the above paragraph again, $L/R(G)\cong T$, for some non-abelian simple group $T$. Since $L\leqslant C$, $LM/R(G)\cong L/R(G)\times M/R(G) \cong S\times T$. Let $F:=\pi (S)\cap \pi(T)$ and $|F|=n$. Then $\Delta(S\times T) \cong K_n \ast \Delta(S)[\rho(S)-F]\ast \Delta(T)[\rho(T)-F]$. Since $\Delta(G)$ is a $K_4$-free graph, $\Delta(S)$ and $\Delta(T)$ are too and $|F|\leqslant 3$. Thus using Lemmas \ref{free} and \ref{chpsl}, we can see that $2,3\in F$, unless when one of the groups $S$ or $T$ is isomorphic to $^2B_2(8)$ or $^2B_2(32)$. If $|F|=1$ or $2$, then $\Delta(S \times T)$ contains a copy of $K_4$ which is a contradiction. Hence $|F|=3$. Thus as $\Delta(G)$ is $K_4$-free, $\rho(S)=\rho(T)=F$. Hence $\Delta(S\times T)$ is a triangle. By Lemma \ref{her}, we can see that $\pi(S)=\pi(\rm{Aut}(S))$. Thus as $S\leqslant G/C\leqslant \rm{Aut}(S)$, $\Delta(G)[\pi(G/C)]$ is a triangle. Hence by Lemma \ref{diamtwo}, $\rm{diam} \Delta(G)\leqslant 2$ which is impossible.
\end{proof}

 Let $G$ be a finite group, $|\rho(G)|\neq 5$ and $\Delta(G)$ be a $K_4$-free graph with diameter three. By Lemma \ref{as}, $|\rho(G)|=6$ and there exists a normal subgroup $R(G)<M\leqslant G$ so that $G/R(G)$ is an almost simple group with socle $S:=M/R(G)$.
 Since $\Delta(G)$ is $K_4$-free, $\Delta(S)$ is too. The structure of $S$ is determined by  Lemma \ref{free}. In the sequel of this paper, we use the notation $\pi_R$ for the set $\rho(R(G))-\pi(G/R(G))$.

\begin{lemma}\label{3p6}
 $|\pi(S)|\neq 3$.
\end{lemma}

\begin{proof}
On the contrary, suppose that $|\pi(S)|=3$. Using Lemma \ref{her}, we can see that $\pi(G/R(G))=\pi(S)$. Thus $|\pi_R|=3$. Hence by Lemma \ref{part}, there exist $p_1,p_2\in \pi_R$ so that $p_1$ and $p_2$ are adjacent vertices in $\Delta(R(G))\subseteq \Delta(G)$. Also using Lemma \ref{part}, $\Delta(G)[\pi(S)]$ is a non-empty graph. Hence by Lemma \ref{good}, we deduce that the induced subgraph of $\Delta(G)$ on $\pi(S)\cup \{p_1,p_2\}$ contains a copy of $K_4$ which is impossible.
\end{proof}

\begin{lemma}\label{alter}
 $S\ncong A_8$.
\end{lemma}

\begin{proof}
On the contrary, suppose that $S\cong A_8$. Since $\Delta(G)$ is $K_4$-free, we can see that $G/R(G)=S$. Thus as $|\rho(G)|= 6$, $|\pi_R|= 2$ and for some distinct primes $p$ and $q$, $\pi_R=\{p,q\}$. Now let $b\in \pi_R$. There exists $\theta_b\in \rm{Irr}(R(G))$ so that $b$ divides $\theta_b(1)$. Suppose $I_b:=I_G(\theta_b)$ and $N_b:=I_b/R(G)$. If $N_b=S$, then looking at the character table of the Schur representation of $S$, we deduce that $70 \theta_b(1)$ or $24 \theta_b(1)$ is a character degree in $\rm{cd}(G|\theta_b)$. Hence $\Delta(G)[\pi(70b)]$ or $\Delta(G)[\pi(S)]$ is a copy of $K_4$, respectively. It is a contradiction. Thus $N_b<S$ is contained in a maximal subgroup $M_b$ of $S$. Hence one of the following cases occurs: \\
Case 1. $[S:M_b]=8$. Then $M_b\cong A_7$. Let $ N_b=M_b\cong A_7$.  Looking at the character table of the Schur representation $\Gamma$ of $A_7$, there exists $m\in \rm{cd}(G|\theta_b)$ such that $m$ is divisible by $6$. Hence the induced subgraph of $\Delta(G)[\pi(S)]\subseteq \Delta(G)$ is a copy of $K_4$ and  it is a contradiction. Thus $N_b$ is a proper subgroup of $M_b$. There exists a maximal subgroup	$L_b$ of $M_b$ containing $N_b$. Now one of the following subcases occurs:\\
a) $[M_b:L_b]=7$. Then using Clifford's Theorem, all character degrees in $\rm{cd}(G|\theta_b)$ are divisible by $7b$. \\
b) $[M_b:L_b]\in \{15,21,35\}$. Then by Clifford's Theorem, there exists $m \in \rm{cd}(G|\theta_b)$ so that $|\pi(m)|\geqslant 4$. It is again a contradiction. \\
Case 2. $[S:M_b]=15$. Then $M_b\cong D\rtimes \rm{PSL}_3(2)$, where $D$ is a 2-subgroup of $M_b$ of order $8$. If $D\nsubseteq N_b$, then using Clifford's Theorem, $2.3.5.b$ divides some character degree in $\rm{cd}(G|\theta_b)$ which is a contradiction. Therefore $D\subseteq N_b$. Now one of the following subcases occurs:\\
a) $N_b/D$ is isomorphic to a proper subgroup of $\rm{PSL}_3(2)$. The index for a maximal subgroup of $\rm{PSL}_3(2)$ is $7$ or $8$. Hence using Clifford's Theorem, there exists $m \in \rm{cd}(G|\theta_b)$ such that $m$ is divisible by either $2.3.5.b$ or $3.5.7.b$. It is a contradiction.\\
b) $N_b/D\cong \rm{PSL}_3(2)$. Then $N_b=M_b$. There exists a normal subgroup $E$ of $I_b$ so that $E/R(G)=D$. Let $\varphi \in \rm{Irr}(E|\theta_b)$, $I_\varphi:=I_{I_b}(\varphi)$ and $N_\varphi:=I_\varphi /E$. If $N_\varphi\cong \rm{PSL}_3(2)$, then $I_\varphi=I_b$ and looking at the character table of the Schur representation of $\rm{PSL}_3(2)$, $8\varphi(1) \in \rm{cd}(I_\varphi|\varphi)$. Therefore using Clifford's Theorem, $m:=2^3.3.5.\varphi(1) \in \rm{cd}(G|\theta_b)$ and it is a contradiction. Hence $N_{\varphi}$ is isomorphic to a proper subgroup of $ \rm{PSL}_3(2)$. Thus by Clifford's Theorem, all character degrees in $\rm{cd}(I_b|\varphi)$ are divisible by $2b$ or $7b$. Therefore by Clifford's Theorem, there exists $d\in \rm{cd}(G|\theta_b)$ such that $d$ is divisible by either $2.3.5.b$ or $3.5.7.b$. It is again a contradiction.\\
Case 3. $[S:M_b]=28,35$ or $56$. Then using Clifford's Theorem, all character degrees in $\rm{cd}(G|\theta_b)$ are divisible by $7b$.\\
Therefore as $\pi_R=\{p,q\}$, by above cases, $7$ is adjacent to both $p$ and $q$. Hence $\rm{diam}(\Delta(G))\leqslant 2$ which is a contradiction.
\end{proof}

\begin{lemma}\label{mat}
$S\ncong M_{11}$.
\end{lemma}

\begin{proof}
On the contrary, suppose that $S\cong M_{11}$. Evidently, $G/R(G)=S$. Thus as $|\rho(G)|= 6$, $|\pi_R|= 2$ and for some distinct primes $p$ and $q$, $\pi_R=\{p,q\}$. Let $b\in \pi_R$.  There exists $\theta_b\in \rm{Irr}(R(G))$ so that $b$ divides $\theta_b(1)$. Suppose $I_b:=I_G(\theta_b)$ and $N_b:=I_b/R(G)$. If $N_b=S$, then as the Schur multiplier of $S$ is trivial, Gallagher's Theorem implies that $\Delta(G)[\pi(110b)]$ is a copy of $K_4$ and it is a contradiction.
 Thus there exists a maximal subgroup $M_b$ of $S$ such that $N_b$ is contained in $M_b$. Now one of the following cases occurs:\\
Case 1. $[S:M_b]=11$. Then $M_b\cong M_{10}$. We first assume that $N_b\cong M_{10}$ or $A_6$. There exists a normal subgroup $J_b$ of $I_b$ such that $L_b:=J_b/R(G)\cong \rm{PSL}_2(9)$. Note that $\theta_b$ is $J_b$-invariant. Thus looking at the character table of
the Schur representation of $L_b$, we deduce that some $m \in \rm{cd}(I_b|\theta_b)$ is divisible by either $6b$ or $10b$. Therefore using Clifford's Theorem, for some $\chi \in \rm{Irr}(G|\theta_b)$, $\chi(1)$ is divisible by $2.3.11.b$ or $2.5.11.b$ which is a contradiction. Now suppose $N_b$ is isomorphic to a proper subgroup of $A_6$.
      Then using Clifford's Theorem, for some $d\in \rm{cd}(G|\theta_b)$, $d$ is divisible by $2.3.11.b$ or $2.5.11.b$. It is again a contradiction. Therefore there exists a maximal subgroup $U_b\ncong A_6$ of $M_b$ containing $N_b$. The index $[M_b:U_b]$ is one of the values $45,36$ or $10$. Thus using Clifford's Theorem, we again obtain a contradiction.\\
Case 2. $[S:M_b]=12$. Then by Clifford's Theorem, $\Delta(G)[\pi(6b)]\cong K_3$. \\
Case 3. $[S:M_b]=66$ or $165$. Then by Clifford's Theorem, all character degrees in $\rm{cd}(G|\theta_b)$ are divisible by either $2.3.11.b$ or $3.5.11.b$ and we again obtain a contradiction.\\
Case 4. $[S:M_b]=55$. Then $M_b$ is non-abelian and $|M_b|=2^4.3^2$. If $N_b$ is a proper subgroup of $M_b$, then using Clifford's Theorem, some $m\in \rm{cd}(G|\theta_b)$ is divisible by $2.5.11.b$ or $3.5.11.b$ which is impossible. Thus $N_b=M_b$. It is easy to see that for some $m\in \rm{cd}(I_b|\theta_b)$, $m$ is divisible by $2b$ or $3b$ and by Clifford's Theorem, we again obtain a contradiction. \\
Therefore by above cases, for every $b\in \pi_R$, $\Delta(G)[\pi(6b)]\cong K_3$. Hence $\rm{diam}\Delta(G)\leqslant 2$ and it is a contradiction.
\end{proof}

\begin{lemma}\label{simple}
 $S\ncong \rm{PSL}_3(4)$, $\rm{PSL}_3(8)$, $\rm{PSU}_3(4)$, $\rm{PSU}_3(9)$,  $^2B_2(8)$ and $^2B_2(32)$.
\end{lemma}

\begin{proof}
On the contrary, suppose that $S$ is isomorphic to one of the above mentioned simple groups. Since $\Delta(G)$ is $K_4$-free, we can see that $\pi(G/R(G))=\pi(S)$. Thus as $|\rho(G)|= 6$, $|\pi_R|= 2$ and for some distinct primes $p$ and $q$, $\pi_R=\{p,q\}$. By Lemma \ref{good}, we deduce that $p$ and $q$ are non-adjacent vertices in $\Delta(G)$. Let $b\in \pi_R$. There exists $\theta_b\in \rm{Irr}(R(G))$ so that $b$ divides $\theta_b(1)$. Suppose $I_b:=I_M(\theta_b)$ and $N_b:=I_b/R(G)$. Then either $N_b=S$ or $N_b$ is contained in a maximal subgroup of $S$. Hence either $\theta_b$ is $M$-invariant or there exists a maximal subgroup $M_b$ of $S$ such that  every $m \in \rm{cd}(M|\theta_b)$ is divisible by $b[S:M_b]$. If $\theta_b$ is $M$-invariant, then looking at the character table of the Schur representation of $S$, we deduce that $\Delta(G)[\pi(S) \cup\{b\}]$ contains a copy  of $K_4$ which is impossible. Note that by Lemma \ref{part}, $\Delta(G)^c$ is bipartite. Therefore  using Clifford's Theorem and the indices of maximal subgroups of $S$, we deduce that $\rm{diam}(\Delta(G))\leqslant 2$ or $\Delta(G)$ contains a copy of $K_4$ which is impossible.
\end{proof}

\begin{lemma}\label{janco}
If $S\cong J_1$, then for some abelian group $A$, $G\cong J_1 \times A$.
\end{lemma}

\begin{proof}
By Lemma \ref{gej}, it is enough to show that $\Delta(G)=\Delta(J_1)$. On the contrary, we assume that $\Delta(G)\neq \Delta(J_1)$. Since $\Delta(G)$ is a $K_4$-free graph with diameter $3$, there exists $(a,b)\in \{(3,19), (3,7),(5,19),(5,7)\}$ such that $a$ and $b$ are adjacent vertices in $\Delta(G)$. There exists $\chi \in \rm{Irr}(G)$ so that $ab|\chi(1)$. Now let $\theta$ be a constituent of $\chi_{R(G)}$, $I:=I_G(\theta)$ and $N:=I/R(G)$.
  As $G/R(G)=S$, $N=S$ or $N$ is contained in a maximal subgroup $M$ of $S$. If $N=S$, then as the Schur multiplier of $S$ is trivial, Gallager's Theorem implies that for some $m\in \rm{cd}(G|\theta)$, $|\pi(m)|\geqslant 4$ which is impossible.
   Thus looking at the indices of maximal subgroups of $S$, by Clifford's Theorem, we deduce that either $a$ and $11$ are adjacent vertices in $\Delta(G)$, or  $|\pi(\chi(1))|\geqslant 4$ which is a contradiction.
\end{proof}

In the sequel, it suffices to consider the case $S\cong \rm{PSL}_2(q)$, where for some prime $u$ and positive integer $\alpha$, $q:=u^\alpha\geqslant 11$. Note that the structure of $\Delta(S)$ is determined by  Lemma \ref{chpsl}.

Suppose $q$ is even. Then by Lemma \ref{chpsl}, $\Delta(S)\cong K_1\cup K_1\cup K_2$,  $K_1\cup K_1\cup K_3$, $K_1\cup K_2\cup K_2$ or  $K_1\cup K_2\cup K_3$. Hence as $\rm{dim}(\Delta(G))=3$, using Lemma \ref{direct product}, we should precisely consider the case $\Delta(S)\cong K_1\cup K_1\cup K_2$.

\begin{lemma}\label{one7}
 $\Delta(S)\ncong K_1\cup K_1 \cup K_2$.
\end{lemma}

\begin{proof}
By Lemmas \ref{chpsl} and \ref{interest}, either $S\cong \rm{PSL}_2(16)$ or $ \rm{PSL}_2(2^p)$, where $p\geqslant 5$ is a prime, $2^p-1$ is a Mersenne prime and $|\pi(2^p+1)|=2$. Thus for some $\epsilon \in \{\pm 1\}$, there exists a prime $t$ such that $\pi(q+\epsilon)=\{t\}$ and $|\pi(q-\epsilon)|=2$. Let $S\cong \rm{PSL}_2(2^P)$ and $G\neq M$. Clearly, $G/R(G)=\rm{Aut}(S)$. Thus as $|\pi(\rm{Aut}(S))|=5$ and $|\rho(G)|=6$, there exists a prime $b$ such that $\pi_R= \{b\}$. Thus for some $\theta \in \rm{Irr}(R(G))$, $b|\theta(1)$. Let $I:=I_M(\theta)$ and $N:=I/R(G)$. Since $\Delta(G)$ is $K_4$-free with diameter three, by Lemma \ref{otf}, we deduce that either $N=S$ and $\rm{cd}(M|\theta)=\{m\theta(1)|\,m\in \rm{cd}(S)\}$, or $N$ is a dihedral group of order $2(2^p+1)$ and for every $m\in \rm{cd}(M|\theta)$, $\pi(m)=\{2,t,b\}$. Using Fermat's Lemma, $p\notin \pi(S)$ and so $I_G(\theta)/R(G)=N$ or $N\rtimes \mathbb{Z}_p$. Thus the Schur multiplier of $I_G(\theta)/R(G)$ is trivial and $\theta$ is extendible to $I_G(\theta)$. If $N=S$, then using Lemma \ref{cdp} and Gallagher's Theorem, $m:=\theta(1) p(2^p+1)\in \rm{cd}(G|\theta)$ and $\Delta(G)[\pi(m)]$ contains a copy of $K_4$ which is impossible. Also if $N$ is a dihedral group, then using Clifford's and Gallagher's Theorems, some character degree in $\rm{cd}(G|\theta)$ is divisible by $2tbp$ and it is a contradiction. Hence $|\pi_R|=2$ and for some distinct primes $b_1$ and $b_2$, $\pi_R=\{b_1,b_2\}$. Let $b_1$ and $b_2$ be adjacent vertices in $\Delta(G)$. There exists $\chi \in \rm{Irr}(G)$ such that $b_1b_2$ divides $\chi_1$. Now assume that $\theta\in \rm{Irr}(R(G))$ is a constituent of $\chi_{R(G)}$. By Lemma \ref{fraction}, $b_1b_2$ divides $\theta(1)$. Hence using Lemma \ref{good}, for some $m\in \rm{cd}(M|\theta)$, $|\pi(m)|\geqslant 4$ which is a contradiction. Thus $b_1$ and $b_2$ are non-adjacent vertices in $\Delta(G)$.
 Hence using Lemma \ref{part}, there exist $x\in \pi_R$ and $y\in \pi (q-\epsilon)$ such that $x$ and $y$ are adjacent vertices in $\Delta(G)$.  Thus  for some $\chi \in \rm{Irr}(G)$, $xy$ divides $\chi(1)$. Now let $\theta \in \rm{Irr}(R(G))$ be a constituent of $\chi_{R(G)}$, $I:=I_M(\theta)$ and $N:=I/R(G)$. Then using Lemma \ref{fraction}, $x$ divides $\theta(1)$. If all character degrees in $\rm{cd}(M|\theta)$ are divisible by $2t\theta(1)$, then $2txy$ divides $\chi(1)$ and it is a contradiction. Thus as $\Delta(G)$ is $K_4$-free,  using Lemma \ref{otf}, $N=S$ and by Gallagher's Theorem, $x$ is adjacent to all vertices in $\pi(S)$. Without loss of generality, we can assume that $x=b_1$. Now we consider the vertex $b_2$.
 There exists $\theta_2\in \rm{Irr}(R(G))$ so that $b_2$ divides $\theta_2(1)$. Let $I_2:=I_M(\theta_2)$ and $N_2:=I_2/R(G)$. Then as $\Delta(G)$ is $K_4$-free with diameter three, by Lemma \ref{otf}, for every $m\in \rm{cd}(M|\theta_2)$, $\pi(m)=\{2,t,b_2\}$.
 We claim that in $\Delta(G)$, no prime in $\pi(2(q+\epsilon))$ is adjacent to any prime in $\pi(q-\epsilon)$. On the contrary, we assume that there exist $x\in \pi(2(q+\epsilon))$ and  $y\in \pi(q-\epsilon)$ such that $x$ and $y$ are adjacent vertices in $\Delta(G)$. Then for some $\chi \in \rm{Irr}(G)$, $xy$ divides $\chi(1)$. Now let $\theta \in \rm{Irr}(R(G))$ be a constituent of $\chi_{R(G)}$, $I:=I_M(\theta)$ and $N:=I/R(G)$. If $N=S$, then by Gallagher's Theorem, $\Delta(G)[\pi(2tyb_1)]$ or $\Delta(G)[\pi(xb_1(q+\epsilon))]$ is a copy of $K_4$ which is impossible. Thus as $\Delta(G)$ is $K_4$-free,  using Lemma \ref{otf}, all character degrees in $\rm{cd}(M|\theta)$ are divisible by $2t$. Hence $2yt$ divides $\chi(1)$ and so the induced subgraph of $\Delta(G)$ on $\{2,t,y,b_1\}$ is a copy of $K_4$ which is  impossible. Therefore in $\Delta(G)$, no prime in $\pi(2(q+\epsilon))$ is adjacent to any prime in $\pi(q-\epsilon)$. Thus using Lemma \ref{direct product}, $\rm{diam}(\Delta(G))\leqslant 2$ and it is a contradiction.
\end{proof}

Finally, we assume that $q$ is odd and try to complete the proof of Main Theorem.

\begin{lemma}\label{six6}
 $\Delta(S)\ncong K_1 \cup ((K_1\cup K_1)\ast K_1)$.
\end{lemma}

\begin{proof}
On the contrary, assume that $\Delta(S)\cong K_1 \cup ((K_1\cup K_1)\ast K_1)$. Since $|\rho(G)|=6$, for some distinct primes $p_1$ and $p_2$, $\rho(G)-\pi(S)=\{p_1,p_2\}$. Using Lemmas \ref{lw} and \ref{sen}, and this fact that $\rm{SL}_2(u^\alpha)$ is the Schur representation of $S$,  $p_1$ and $p_2$ are adjacent to all vertices in $\pi(u^{2\alpha}-1)$. Thus as $\rm{diam}(\Delta(G))=3$, there exists unique $x \in \rho(G)-\{2,u\}$ such that $x$ and $u$ are adjacent vertices in $\Delta(G)$. Let $a,b\in \rho(G)-\{x,u\}$ be non-adjacent vertices in $\Delta(G)$. Then as $\rm{diam}(\Delta(G))=3$, the induced subgraph of $\Delta(G)^c$ on $\{a,b,u\}$ is a copy of $K_3$. It is a contradiction with Lemma \ref{part}. Hence the induced subgraph of $\Delta(G)$ on $\rho(G)-\{x,u\}$ is a copy of $K_4$ and it is a contradiction.
\end{proof}

\begin{lemma}\label{seven6}
 $\Delta(S)\ncong K_1\cup K_3$.
\end{lemma}

\begin{proof}
On the contrary, assume that $\Delta(S)\cong K_1\cup K_3$. By Lemma \ref{chpsl}, for some $\epsilon \in \{\pm 1\}$, $\pi(u^\alpha-\epsilon)=\{2\}$ and $|\pi(u^\alpha+\epsilon)|=3$. If $t\in \rho(G)-\pi(S)$ divides $[G:M]$, then by Lemma \ref{lw}, $\Delta(G)[\pi(t(u^\alpha+\epsilon))]\cong K_4$ which is impossible. Hence as $|\pi(S)|=4$ and $|\rho(G)|=6$, $|\pi_R|=2$. Thus there exist distinct primes $p$ and $p^\prime$ so that $\pi_R=\{p,p^\prime\}$. Let $b\in \pi_R$. There exists $\theta_b\in \rm{Irr}(R(G))$ so that $b$ divides $\theta_b(1)$. Suppose $I_b:=I_M(\theta_b)$ and $N_b:=I_b/R(G)$.
 If $N_b=S$, then as $\rm{SL}_2(u^\alpha)$ is the Schur representation of $S$, $\Delta(G)[\pi(b(u^\alpha+\epsilon))]$ is a copy of $K_4$ and we obtain a contradiction.
 Thus by Lemma \ref{sen}, $N_b$ is contained in a dihedral group, or $N_b$ is a Frobenius group whose kernel is an elementary abelian $u$-group and $\theta_b$ is not extendible to $I_b$. Hence as $\Delta(G)$ is $K_4$-free, by Clifford's Theorem and Lemma \ref{Frobenius}, we deduce that some $m\in \rm{cd}(M|\theta_b)$ is divisible by  $2ub$.
Therefore as $2\in \pi(q+\epsilon)$, the vertex $2$ is adjacent to all vertices in $\rho(G)-\{2\}$ and $\rm{diam}\Delta(G)\leqslant 2$. It is a contradiction.
\end{proof}

\begin{lemma}\label{eight6}
 $\Delta(S)\ncong K_1 \cup ((K_1\cup K_2)\ast K_1)$.
\end{lemma}

\begin{proof}
On the contrary, assume that $\Delta(S)\cong K_1 \cup ((K_1\cup K_2)\ast K_1)$. Using Lemma \ref{chpsl}, there exists $\epsilon\in \{\pm1\}$ so that $|\pi(u^\alpha-\epsilon)|=2$ and $|\pi(u^\alpha+\epsilon)|=3$.
 Also as $|\pi(S)|=5$ and $|\rho(G)|=6$, for some prime $p$, $\rho(G)-\pi(S)=\{p\}$. Thus using Lemmas \ref{lw} and \ref{sen}, and this fact that $\rm{SL}_2(u^\alpha)$ is the Schur representation of $S$, $\Delta(G)[\pi(p(u^\alpha+\epsilon))]$ is a copy of $K_4$  and we obtain a contradiction.
\end{proof}

\begin{lemma}\label{nine6}
$\Delta(S)\ncong K_1 \cup ((K_2\cup K_2)\ast K_1)$.
\end{lemma}

\begin{proof}
On the contrary, assume that $\Delta(S)\cong K_1 \cup ((K_2\cup K_2)\ast K_1)$. Using Lemma \ref{chpsl}, $|\pi(u^\alpha-1)|=|\pi(u^\alpha+1)|=3$. Also by Lemma \ref{lw},  no prime in $\pi(S)-\{2\}$ is in $\pi([G:M])$. Since $\Delta(S)\cong K_1 \cup ((K_2\cup K_2)\ast K_1)$ and $\rm{diam}\Delta(G)=3$, there exists $x\in \pi(u^{2\alpha}-1)-\{2\}$ in which $u$ and $x$ are adjacent vertices in $\Delta(G)$.
 Thus there exists $\chi\in \rm{Irr}(G)$ so that $ux$ divides $\chi(1)$. Let $\varphi\in \rm{Irr}(M)$ and $\theta\in \rm{Irr}(R(G))$ be constituents of $\chi_M$ and $\varphi_{R(G)}$, respectively. Then using Lemma \ref{fraction} and this fact that no prime in $\pi(S)-\{2\}$ divides $[G:M]$, we deduce that $ux$ divides $\varphi(1)$. Note that for some $\epsilon \in \{\pm1\}$, $x$ divides $u^\alpha+\epsilon$. By Lemma \ref{sen}, $\theta$ is $M$-invariant and $\varphi(1)\in \rm{cd}(M|\theta)$. Hence as $\rm{SL}_2(u^\alpha)$ is the Schur representation of $S$, the induced subgraph of $\Delta(G)$ on $\pi(u(u^\alpha-\epsilon))$ or $\pi(x(u^\alpha-\epsilon))$ is a copy of $K_4$ which is impossible.
 \end{proof}
%%%%%%%%%%%%%%%%%%%%%%%%%%%%%%%%%%%
\section*{Acknowledgements}
This research was supported in part
by a grant  from School of Mathematics, Institute for Research in Fundamental Sciences (IPM).

 I would like to express
my gratitude to dear Prof. Gerhard Hiss and dear Prof. Mark Lewis for valuable comments which
improved my manuscript.


\begin{thebibliography}{22}
\bibitem{AC}
 Z. Akhlaghi, C. Casolo, S. Dolfi, E. Pacifici and L. Sanus. On the character degree graph of finite groups.  Anali di Mat. Pura Appl.  198 (2019), 1595-1614.
%
\bibitem{[SDM]}
C. Casolo, S. Dolfi, E. Pacifici and L. Sanus. Groups whose character degree graph has diameter three. Israel J. Math. 215 (2016), 523-558.
%
\bibitem{[At]}
J. H. Conway, R. T. Curtis, S. P. Norton, R. A. Parker, R. A. Wilson. Atlas of finite groups (Oxford University Press, London, 1984).
%
\bibitem{commu}
M. Ebrahimi. $K_4$-free character graphs with seven vertices.  Comm. Algebra.  48 (3) (2020), 1001-1010.
%
\bibitem{ams}
M. Ebrahimi. Character graphs with diameter three. Proc. Amer. Math. Soc. (to appear)
%
\bibitem{[her]}
M. Herzog. On finite simple groups of order divisible by three primes only. J. Algebra. 10 (1968), 383-388.
%
\bibitem{hup}
B. Huppert. Endliche gruppen I  (Springer-Verlag, Berlin, 1983).
%
\bibitem{hup2}
B. Huppert and W. Lempken. Simple groups of order divisible by at most four primes. Proc. F. Scorina Gomel State Univ. 16 (3) (2000), 64-75.
%
\bibitem{[isa]}
I. M. Isaacs. Character theory of finite groups (Academic Press. San Diego, 1976).
%
\bibitem{rep}
J. C. Jantzen. Representations of algebraic groups (American Mathematical Soc., 2007).
%
\bibitem{[A]}
M. L. Lewis. Solvable groups with character degree graphs having 5 vertices and diameter 3. Comm. Algebra. 30 (2002),
5485-5503.
%
\bibitem{[M]}
M. L. Lewis. An overview of graphs associated with character
degrees and conjugacy class sizes in finite group. Rocky
Mountain, J. Math. 38 (1) (2008), 175-211.
%
\bibitem{[LMc]}
M. L. Lewis and J. McVey. Character degree graphs of automorphism groups of characteristically simple groups. J. Group Theory. 12 (3) (2009), 387-391.
%
\bibitem{[nsc]}
M. L. Lewis and D. L. White. Connectedness of degree graphs of non-solvable groups. J. Algebra. 266(1) (2003), 51-76.
%
\bibitem{[DM]}
M. L. Lewis and D. L. White. Diameters of degree graphs of non-solvable groups II. J. Algebra. 312 (2007), 634-649.
%
\bibitem{[non]}
M. L. Lewis and D. L. White. Non-solvable groups with no prime dividing three character degrees. J. Algebra. 336 (2011), 158-183.
%
\bibitem{[I]}
O. Manz, R. Staszewski and W. Willems, On the number of components
of a graph related to character degrees. Proc. Amer.
Math. Soc. 103 (1) (1988), 31-37.
%
\bibitem{[P]}
O. Manz, W. Willems and T. R. Wolf. The diameter of the character
degree graph. J. Reine angew. Math. 402 (1989),
181-198.
%
\bibitem{[Ton]}
H. P. Tong-Viet. Groups whose prime graphs have no triangles. J. Algebra. 378 (2013), 196-206.
%
\bibitem{finite}
H. P. Tong-viet. Finite groups whose prime graphs are regular. J. Algebra. 397 (2014), 18-31.
%
\bibitem{Tong}
H. P. Tong-viet. Characterization of some simple groups by the multiplicity pattern. Monatsh. Math. 172 (2013), 189-206.
%
\bibitem{[white]}
D. L. White. Degree graphs of simple linear and unitary groups. Comm. Algebra. 34 (8) (2006), 2907-2921.
%
\bibitem{zhang}
J. Zhang. On a problem by huppert. Acta Scientiarum Naturalium Universitatis Pekinensis. 34 (1998), 143-150.
%
\bibitem{zsi}
K. Zsigmondy. Zur theorie der potenzreste. Monathsh. Math. Phys. 3(1) (1892), 265-284.

\end{thebibliography}
\end{document}